\documentclass[12pt]{article}

\usepackage{amsmath,amssymb,amsthm,fullpage}

\newtheorem{theorem}{Theorem}[section]

\newtheorem{corollary}[theorem]{Corollary}
\newtheorem{lemma}[theorem]{Lemma}

\numberwithin{equation}{section}

\def\C{{\mathbb C}}
\def\D{{\mathbb D}}

\def\N{{\mathbb N}}

\def\R{{\mathbb R}}

\def\E{{\mathbb E}}
\def\T{{\mathbb T}}
\def\P{{\mathbb P}}
\newcommand{\cA}{\mathcal{A}}
\newcommand{\cE}{\mathcal{E}}
\def\eps{\varepsilon}

\begin{document}

\title{\bf Zero Distribution of Random Polynomials}

\author{Igor E. Pritsker}

\date{}

\maketitle

\begin{abstract}
We study global distribution of zeros for a wide range of ensembles of random polynomials. Two main directions are related to almost sure limits of the zero counting measures, and to quantitative results on the expected number of zeros in various sets. In the simplest case of Kac polynomials, given by the linear combinations of monomials with i.i.d. random coefficients, it is well known that their zeros are asymptotically uniformly distributed near the unit circumference under mild assumptions on the coefficients.  We give estimates of the expected discrepancy between the zero counting measure and the normalized arclength on the unit circle. Similar results are established for polynomials with random coefficients spanned by different bases, e.g., by orthogonal polynomials. We show almost sure convergence of the zero counting measures to the corresponding equilibrium measures for associated sets in the plane, and quantify this convergence. Random coefficients may be dependent and need not have identical distributions in our results.
\end{abstract}


\section{Introduction}

Zeros of polynomials of the form $P_n(z)=\sum_{k=0}^{n} A_k z^k,$ where $\{A_n\}_{k=0}^n$ are random coefficients,  have been studied by Bloch and P\'olya, Littlewood and Offord, Erd\H{o}s and Offord, Kac, Rice, Hammersley, Shparo and Shur, Arnold, and many other authors. The early history of the subject with numerous references is summarized in the books by Bharucha-Reid and Sambandham \cite{BR}, and by Farahmand \cite{Fa}. It is well known that, under mild conditions on the probability distribution of the coefficients, the majority of zeros of these polynomials accumulate near the unit circumference, being equidistributed in the angular sense.  Introducing modern terminology, we call a collection of random polynomials $P_n(z)=\sum_{k=0}^n A_k z^k,\ n\in\N,$ the ensemble of \emph{Kac polynomials}.
Let $\{Z_k\}_{k=1}^n$ be the zeros of a polynomial $P_n$ of degree $n$, and define the \emph{zero counting measure}
\[
\tau_n=\frac{1}{n}\sum_{k=1}^n \delta_{Z_k}.
\]
The fact of equidistribution for the zeros of random polynomials can now be expressed via the weak convergence of $\tau_n$ to the normalized arclength measure $\mu_{\T}$ on the unit circumference, where $d\mu_{\T}(e^{it}):=dt/(2\pi).$ Namely, we have that $\tau_n \stackrel{w}{\rightarrow} \mu_{\T}$ with probability 1 (abbreviated as a.s. or almost surely). More recent work on the global distribution of zeros of Kac polynomials include papers of Hughes and Nikeghbali \cite{HN}, Ibragimov and Zeitouni \cite{IZ}, Ibragimov and Zaporozhets \cite{IZa}, Kabluchko and Zaporozhets \cite{KZ1,KZ2}, etc. In particular, Ibragimov and Zaporozhets \cite{IZa} proved that if the coefficients are independent and identically distributed, then the condition $\E[\log^+|A_0|]<\infty$ is necessary and sufficient for $\tau_n \stackrel{w}{\rightarrow} \mu_{\T}$ almost surely. Here, $\E[X]$ denotes the expectation of a random variable $X$.

The majority of available results require the coefficients $\{A_k\}_{k=0}^n$ be independent and identically distributed (i.i.d.) random variables. This assumption is certainly natural from probabilistic point of view. However, it is not necessary as the following simple example shows. If $A_k=\xi,\ k=0,1,2,\ldots,$ are identical (hence dependent), where $\xi$ is a complex random variable, then we deal with the family of polynomials
\[
P_n(z)=\sum_{k=0}^n \xi z^k = \xi\, \frac{z^{n+1}-1}{z-1}, \quad n\in\N.
\]
The zeros of $P_n$ are uniformly distributed on $\T,$ being the $n+1$-st roots of unity except $z=1$. Furthermore, $\tau_n \stackrel{w}{\rightarrow} \mu_{\T}$ almost surely, provided $\xi$ does not vanish with positive probability. The assumption of identical distribution for coefficients is not necessary for $\tau_n \stackrel{w}{\rightarrow} \mu_{\T}$ a.s. either. Thus one of our main goals is to remove unnecessary restrictions, and prove results on zeros of polynomials whose coefficients need not have identical distributions and may be dependent. 

Another interesting direction is related to the study of zeros of random polynomials spanned by various bases, e.g., by orthogonal polynomials. These questions were considered by Shiffman and Zelditch \cite{SZ1}-\cite{SZ3}, Bloom \cite{Bl1} and \cite{Bl2}, Bloom and Shiffman \cite{BS}, Bloom and Levenberg \cite{BL}, Bayraktar \cite{Ba} and others. It is of importance for us that many mentioned papers used potential theoretic approach to study the limiting zero distribution, including that of multivariate polynomials. We develop such ideas here, and extend them to more general bases and classes of random coefficients, but only for the univariate case.

We do not discuss the local scaling limit results on the zeros of random polynomials as this falls beyond the scope of the paper. Instead, we direct the reader to recent interesting papers on this topic by Tao and Vu \cite{TV}, and by Sinclair and Yattselev \cite{SY}.

The rest of our paper is organized as follows. Section 2 deals with almost sure convergence of the zero counting measures for polynomials with random coefficients that satisfy only weak $\log$-integrability assumptions. Section 3 develops the discrepancy results of \cite{PS} and \cite{PY}, and establishes expected rates of convergence of the zero counting measures to the equilibrium measures. Again, the random coefficients in Section 3 are neither independent nor identically distributed, and their distributions only satisfy certain uniform bounds for the fractional and logarithmic moments. We also consider random polynomials spanned by general bases in Sections 2 and 3, which includes random orthogonal polynomials and random Faber polynomials on various sets in the plane. All proofs are given in Section 4.

\section{Asymptotic Equidistribution of Zeros}

We first study the limiting behavior of the normalized zero counting measures for sequences of polynomials of the form
\[
P_n(z)=\sum_{k=0}^n A_k z^k, \quad n\in\N.
\]
Let $A_k,\ k=0,1,2,\ldots,$ be complex valued random variables that are not necessarily independent, nor they are required to be identically distributed. Recall that the distribution function of $|A_k|$ is defined by $F_k(x)=\P(\{|A_k|\le x\}),\ x\in\R,$ see Gut \cite[Section 2.1]{Gut}. We use the following assumptions on random coefficients in this section.

\textbf{Assumption 1} There is $N\in\N$ and a decreasing function $f:[a,\infty)\to[0,1],\ a>1,$ such that
\begin{align} \label{2.1}
\int_a^{\infty} \frac{f(x)}{x}\,dx < \infty \mbox{ and } 1-F_k(x)\le f(x),\ \forall\, x\in[a,\infty), 
\end{align}
holds for all $k=0,1,2,\ldots.$

\textbf{Assumption 2} There is $N\in\N$ and an increasing function $g:[0,b]\to[0,1],\ 0<b<1,$ such that
\begin{align} \label{2.2}
\int_0^b \frac{g(x)}{x}\,dx < \infty  \mbox{ and } F_k(x)\le g(x),\ \forall\, x\in[0,b],
\end{align}
holds for all $k=0,1,2,\ldots.$

If $F(x)$ is the distribution function of $|X|$, where $X$ is a complex random variable, then 
\[
\E[\log^+|X|]<\infty \quad \Leftrightarrow \quad \int_a^{\infty} \frac{1-F(x)}{x}\,dx < \infty,\ a\ge 0,
\]
and
\[
\E[\log^-|X|]<\infty \quad \Leftrightarrow \quad \int_0^b \frac{F(x)}{x}\,dx < \infty,\ b>0,
\]
see, e.g., Theorem 12.3 of Gut \cite[p. 76]{Gut}. Hence when all random variables $|A_k|,\ k=0,1,\ldots,$ are identically distributed, one can state assumptions \eqref{2.1}-\eqref{2.2} in a more compact equivalent form
\[
\E[|\log|A_0||]<\infty.
\]
Assumption \eqref{2.2} readily implies that $\P(\{A_k=0\})=0$ for all $k$, i.e., the probability measures of the coefficients cannot have point masses at $0$. But they can have point masses elsewhere, and need not possess densities. 

Schehr and Majumdar \cite{SM} considered random polynomials with Gaussian coefficients $A_k$ that have mean zero and variance $\sigma_k^2=e^{-k^\alpha}$, and found that the expected number of real zeros for $P_n(z)=\sum_{k=0}^n A_k z^k$ is asymptotic to $n$ for $\alpha>2$. Thus almost sure equidistribution of zeros near the unit circumference can clearly fail in absence of uniform assumptions on coefficients.

We show in Lemma \ref{lem4.2} of Section \ref{secP} that if both \eqref{2.1} and \eqref{2.2} hold, then 
\[
\lim_{n\to\infty} |A_0|^{1/n} = \lim_{n\to\infty} |A_n|^{1/n} = \lim_{n\to\infty} \max_{0\le k\le n} |A_k|^{1/n} = 1 \quad\mbox{ a.s.}
\]
These facts allow to apply potential theoretic techniques developed to study the asymptotic zero distribution of deterministic polynomials (see Andrievskii and Blatt \cite{AB} for an overview). We start with the following result for the Kac ensemble.

\begin{theorem} \label{thm2.1}
If the coefficients of $P_n(z)=\sum_{k=0}^n A_k z^k,\ n\in\N,$ are complex random variables that satisfy assumptions \eqref{2.1} and \eqref{2.2}, then the zero counting measures $\tau_n$ for this sequence of polynomials converge almost surely to $\mu_{\T}$ as $n\to\infty$.
\end{theorem}

We next consider more general ensembles of random polynomials 
$$P_n(z) = \sum_{k=0}^n A_k B_k(z)$$
spanned by various bases $\{B_k\}_{k=0}^{\infty}.$ Let $B_k(z)=\sum_{j=0}^k b_{j,k}z^j$, where $b_{j,k}\in\C$ for all $j$ and $k$, and $b_{k,k}\neq 0$ for all $k$, be a polynomial basis, i.e., a linearly independent set of polynomials. Note that $\text{deg}\  B_k=k$ for all $k\in\N\cup\{0\}.$ Given a compact set $E\subset\C$ of positive logarithmic capacity $\text{cap}(E)$ (cf. Ransford \cite{Ra}), we assume that
\begin{align} \label{2.3}
\limsup_{k\to\infty} \|B_k\|_E^{1/k} \le 1 \quad \mbox{and} \quad \lim_{k \rightarrow \infty} |b_{k,k}|^{1/k}=1/\textup{cap}(E),
\end{align}
where $\|B_k\|_E:=\sup_E|B_k|.$ Condition \eqref{2.3} holds for many standard bases used for representing analytic functions on $E$, e.g., for various sequences of orthogonal polynomials (cf. Stahl and Totik \cite{ST}) and for Faber polynomials (see Suetin \cite{Su}). In the former case, random polynomials spanned by such bases are called random orthogonal polynomials. Their asymptotic zero distribution was recently studied in a series of papers by Shiffman and Zelditch \cite{SZ2}, Bloom \cite{Bl1} and \cite{Bl2}, Bloom and Shiffman \cite{BS}, Bloom and Levenberg \cite{BL} and Bayraktar \cite{Ba}. In particular, it was shown that the counting measures of zeros converge weakly to the equilibrium measure of $E$ denoted by $\mu_E$, which is a positive unit Borel measure supported on the outer boundary of $E$ \cite{Ra}. Most of mentioned papers also considered multivariate polynomials. They assumed that the basis polynomials are orthonormal with respect to a measure satisfying the Bernstein-Markov property, and that the coefficients are complex i.i.d. random variables with uniformly bounded distribution density function with respect to the area measure, and with proper decay at infinity. 

We develop this line of research by using the results of Blatt, Saff and Simkani \cite{BSS} for deterministic polynomials of a single variable. In particular, we relax conditions on the random coefficients and consider more general choices of the bases.

\begin{theorem} \label{thm2.2}
Suppose that a compact set $E\subset\C,\ \textup{cap}(E)>0,$ has empty interior and connected complement. If the coefficients $A_k$ satisfy \eqref{2.1}-\eqref{2.2}, and the basis polynomials $\{B_k\}_{k=0}^{\infty}$ satisfy \eqref{2.3}, then the zero counting measures of $P_n(z) = \sum_{k=0}^n A_k B_k(z)$ converge almost surely to $\mu_E$ as $n\to\infty$.
\end{theorem}

Two most interesting applications of this result are related to random orthogonal  and random Faber polynomials. Orthogonality below is considered with respect to the weighted arclength measure $w(s)\,ds$ on $E,$ and the definition of Faber polynomials may be found in Chapter 2 of \cite{Su}.

\begin{corollary} \label{cor2.3}
Assume that conditions \eqref{2.1}-\eqref{2.2} hold for the coefficients.\\
\textup{(i)} Suppose that $E$ is a finite union of rectifiable Jordan arcs with connected complement. If the basis polynomials $B_k$ are orthonormal with respect to a positive Borel measure $\mu$ supported on $E$ such that the Radon-Nikodym derivative $w(s)=d\mu/ds>0$ for almost every $s,$ then \eqref{2.3} is satisfied and $\tau_n$ converge almost surely to $\mu_E$ as $n\to\infty$.\\
\textup{(ii)} Suppose that $E$ is a compact connected set with empty interior and connected complement, and that $E$ is not a single point. If the basis polynomials $B_k$ are
the Faber polynomials of $E$, then \eqref{2.3} holds true and $\tau_n$ converge almost surely to $\mu_E$ as $n\to\infty$. 
\end{corollary}

If the interior of $E$ is not empty, we often need extra conditions to prevent excessive accumulation of zeros there. However, these additional assumptions may be replaced by more specific choices of the basis and geometric properties of $E$ as in the following result. If $E$ is a finite union of rectifiable curves and arcs, we call the polynomials orthonormal with respect to the arclength measure $ds$ by Szeg\H{o} polynomials. When $E$ is a compact set of positive area (2-dimensional Lebesgue measure on $\C$), we call the polynomials orthonormal with respect to the area measure $dA$ on $E$ by Bergman polynomials.

\begin{theorem} \label{thm2.4}
Suppose that $E$ is the closure of a Jordan domain with analytic boundary, and that the basis $\{B_k\}_{k=0}^{\infty}$ is given either by Szeg\H{o}, or by Bergman, or by Faber polynomials. If \eqref{2.1}-\eqref{2.2} hold for the coefficients $A_k$, then the zero counting measures of $P_n(z) = \sum_{k=0}^n A_k B_k(z)$ converge almost surely to $\mu_E$ as $n\to\infty$.
\end{theorem}

In a more general setting, we introduce an extra assumption \eqref{2.4} on the constant term $A_0$.

\begin{theorem} \label{thm2.5}
Let $E\subset\C$ be any compact set of positive capacity. If \eqref{2.1}-\eqref{2.3} hold, $A_0$ is independent from $\{A_n\}_{n=1}^\infty,$ and there is $t>1$ such that 
\begin{align} \label{2.4}
\sup_{z\in\C}\, \E\left[(\log^-|A_0-z|)^t\right]<\infty,
\end{align}
then the zero counting measures of $P_n(z) = \sum_{k=0}^n A_k B_k(z)$ converge almost surely to $\mu_E$ as $n\to\infty$.
\end{theorem}
Assumption \eqref{2.4} means that the probability measure of $A_0$ cannot be too concentrated at any point $z\in\C.$ In particular, it rules out the possibility that $A_0$ takes any specific value with positive probability, so that $A_0$ cannot be a discrete random variable. On the other hand, if $A_0$ is a continuous random variable satisfying \eqref{2.4}, its density need not be bounded. For example, if the probability measure $\nu$ of $A_0$ is absolutely continuous with respect to the area measure $dA$ and has density $d\nu/dA(w)$ uniformly bounded by $C/|w-z|^s,\ s<2,$ near every $z\in\C,$ then \eqref{2.4} holds.

\medskip
Since we used a sequence of random coefficients $\{A_k\}_{k=0}^{\infty},$ polynomials $\sum_{k=0}^n A_k B_k(z)$ were essentially partial sums of a random series. We now discuss even more general sequences of random polynomials of the form
$$P_n(z) = \sum_{k=0}^n A_{k,n} B_k(z).$$
Here we deal with a triangular array of coefficients $A_{k,n},\ k=0,1,\ldots,n,\ n\in\N,$ that are complex valued random variables. As before, they need not be identically distributed. We denote the distribution function of $|A_{k,n}|$ by $F_{k,n}.$ Assumptions 1 and 2 uniformly imposed on all coefficients $A_{k,n}$ suffice to obtain that
\[
\lim_{n\to\infty} |A_{0,n}|^{1/n} = \lim_{n\to\infty} |A_{n,n}|^{1/n} = 1 \quad\mbox{ a.s.}
\]
by Lemma \ref{lem4.1}. But we need a slightly stronger condition to prove the limit
\[
\lim_{n\to\infty} \max_{0\le k\le n} |A_{k,n}|^{1/n} = 1 \quad\mbox{ a.s.}
\]
Thus we introduce the following assumptions on the triangular array of random coefficients.

\textbf{Assumption 1*} There is $N\in\N$ such that $\{|A_{k,n}|\}_{k=0}^n$ are jointly independent for each $n\ge N.$ Furthermore, there is a function $f:[a,\infty)\to[0,1],\ a>1,$ such that $f(x)\log{x}$ is decreasing, and
\begin{align} \label{2.5}
\int_a^{\infty} f(x)\,\frac{\log{x}}{x}\,dx < \infty \mbox{ and } 1-F_{k,n}(x)\le f(x),\ \forall\, x\in[a,\infty), 
\end{align}
holds for all $k=0,1,\ldots,n,$ and all $n\ge N.$

\textbf{Assumption 2*} There is $N\in\N$ and an increasing function $g:[0,b]\to[0,1],\ 0<b<1,$ such that
\begin{align} \label{2.6}
\int_0^b \frac{g(x)}{x}\,dx < \infty  \mbox{ and } F_{k,n}(x)\le g(x),\ \forall\, x\in[0,b],
\end{align}
holds for all $k=0,1,\ldots,n,$ and all $n\ge N.$

Lemma \ref{lem4.3} in Section 4 gives all necessary limits \eqref{4.11}-\eqref{4.13} that allow to extend Theorems \ref{thm2.1}, \ref{thm2.2} and Corollary \ref{cor2.3} by following similar ideas, but certainly replacing \eqref{2.1} and \eqref{2.2} (Assumptions 1 and 2) with \eqref{2.5} and \eqref{2.6} (Assumptions 1* and 2*). The corresponding analog of Theorem \ref{2.5} also holds if we replace \eqref{2.1} and \eqref{2.2} by \eqref{2.5} and \eqref{2.6}, as well as replace \eqref{2.4} by the condition
\begin{align} \label{2.7}
\limsup_{n\to\infty}\,\sup_{z\in\C} \E\left[(\log^-|A_{0,n}-z|)^t\right]<\infty,
\end{align}
for a fixed $t>1.$ Detailed proofs of these statements may be found in \cite{Pr3}, and we confine ourselves to an outline of the necessary arguments in this paper.

\section{Expected Number of Zeros of Random Polynomials}

Results of this section provide quantitative estimates for the weak convergence of the zero counting measures of random polynomials to the corresponding equilibrium measures. In particular, we study the expected deviation of the normalized counting measure of zeros $\tau_n$ from the equilibrium measure $\mu_E$ on certain sets, which is often referred to as discrepancy between those measures. We again assume that the complex valued random variables $A_k,\ k=0,1,2,\ldots,$ are not necessarily independent nor identically distributed. It is convenient to first discuss the simplest case of the unit circle, which originated in \cite{PS}. A standard way to study the deviation of $\tau_n$ from $\mu_{\T}$ is to consider the discrepancy of these measures in the annular sectors of the form
$$\cA_r(\alpha,\beta)=\{z\in \C :r<|z|<1/r, \ \alpha \leq \text{arg} \ z <\beta \}, \quad 0<r<1.$$
The recent paper of Pritsker and Yeager \cite{PY} contains the following estimate of the discrepancy.

\begin{theorem} \label{thm3.1}
Suppose that the coefficients of $P_n(z)=\sum_{k=0}^n A_k z^k$ are complex random variables that satisfy:
\begin{enumerate}
\item $\E[|A_k|^t]<\infty,\ k=0,\ldots,n,$ for a fixed $t\in(0,1]$
\item $\E[\log|A_0|] > -\infty$ and $\E[\log|A_n|] > -\infty.$
\end{enumerate}
Then we have for all large $n\in \N$ that
\begin{align} \label{3.1}
\E\left[\left|\tau_n(\cA_r(\alpha,\beta))-\frac{\beta-\alpha}{2\pi}\right|\right] \leq C_r \left[\frac{1}{n}\left(\frac{1}{t}\log \sum_{k=0}^n \E[|A_k|^t] - \frac{1}{2}\E[\log|A_0A_n|] \right)\right]^{1/2},
\end{align}
where
$$C_r := \sqrt{\frac{2\pi}{\mathbf{k}}}+\frac{2}{1-r} \quad\mbox{with}\quad \mathbf{k}:=\sum_{k=0}^{\infty}\frac{(-1)^k}{(2k+1)^2}$$
being Catalan's constant.
\end{theorem}

Introducing uniform bounds, \cite{PY} also provides the rates of convergence for the expected discrepancy as $n\to\infty.$
\begin{corollary} \label{cor3.2}
Let $P_n(z)=\sum_{k=0}^n A_{k,n} z^k,\ n\in\N,$ be a sequence of random polynomials. If
$$M := \sup\{\E[|A_{k,n}|^t]\ \vert \ k=0,\ldots,n,\ n\in\N\} < \infty$$
and
$$L :=  \inf\{\E[\log|A_{k,n}|]\ \vert \ k=0\,\&\,n,\ n\in\N\} > - \infty,$$
then
\begin{align*}
\E\left[\left|\tau_n(\cA_r(\alpha,\beta))-\frac{\beta-\alpha}{2\pi}\right|\right] \leq C_r \left[\frac{1}{n}\left(\frac{\log (n+1)+\log M}{t} - L \right)\right]^{1/2} = O\left(\sqrt{\frac{\log{n}}{n}}\right)
\end{align*}
as $n\to\infty.$
\end{corollary}
It is well known from the original work of Erd\H{o}s and Tur\'an \cite{ET} that the order $\sqrt{\log{n}/n}$ is optimal in the deterministic case. The proofs of Theorem \ref{thm3.1} and Corollary \ref{cor3.2} are sketched in Section 4 for convenience of the reader. Papers \cite{PS} and \cite{PY} explain how one can obtain quantitative results about the expected number of zeros of random polynomials in various sets, see Propositions 2.3-2.5 of \cite{PY}. The basic observation here is that the number of zeros of $P_n$ in a set $S\subset\C$ denoted by $N_n(S)$ is equal to $n\tau_n(S),$ and the estimates for $\E[N_n(S)]$ readily follow from Theorem \ref{thm3.1} and Corollary \ref{cor3.2}. 

\bigskip
We now turn to random polynomials spanned by the general bases $B_k(z)=\sum_{j=0}^k b_{j,k}z^j,\ k=0,1,\ldots$, where $b_{j,k}\in\C$ for all $j$ and $k$, and $b_{k,k}\neq 0$ for all $k$. These bases are considered in conjunction with an arbitrary compact set $E$ of positive capacity in the plane, whose equilibrium measure is denoted by $\mu_E.$ It is known that in order to obtain the discrepancy results for the pair $\tau_n$ and $\mu_E$ on compact sets $E\subset\C$, one inevitably needs to restrict the geometric properties of $E$, see Andrievskii and Blatt \cite{AB}. Although assumption \eqref{2.3} is typically sufficient for the discrepancy to converge to $0$ as $n\to\infty,$ we need a different assumption to obtain the rates of convergence as in Corollary \ref{cor3.2}. In fact, many important bases satisfy
\begin{align} \label{3.2}
\|B_k\|_E = O(k^p) \quad \mbox{and} \quad |b_{k,k}|(\text{cap}(E))^k \ge c\, k^{-q}\quad \mbox{as } k\to\infty,
\end{align}
with fixed positive constants $c,p,q$.

Instead of the annular sectors $\cA_r(\alpha,\beta)$, we use the ``generalized sectors" $\cA_r$ defined with help of the Green function and conformal mappings. As in the previous section, we begin with the case when $E$ has empty interior. Specifically, let $E$ be a compact set with one connected component being a Jordan arc $L$ such that the distance from $L$ to $E\setminus L$ is positive. Denote the Green function of $\overline{\C}\setminus E$ with pole at infinity by $g_E(z)$, and denote its harmonic conjugate by $\tilde{g}_E(z).$ One can find $b_L>0$ such that $\Phi(z)=\exp[b_L(g_E(z)+i\tilde{g}_E(z))]$ defines a conformal bijection between an annular region $U_L$ with inner boundary $L$ and an annulus $1<|w|<R,\ R>1.$ The mapping $\Phi$ extends to $L$ with values in $\T$ by a standard argument. Given any subarc $J\subset L$ and $r\in(1,R)$, we set 
\[
\cA_r=\cA_r(J)=\{z\in \overline{U}_L: 1 \le |\Phi(z)| \le r \mbox{ and }\Phi(z)/|\Phi(z)| \in \Phi(J)\}. 
\]
In other words, $\cA_r$ is a curvilinear strip around $J$ that is bounded by the level curve $|\Phi(z)|=r.$ More details of this construction may be found in Chapter 2 of \cite{AB}. 

A smooth Jordan curve is said to be Dini-smooth if the angle between its tangent line and positive real axis is Dini-continuous as a function of arclength parameter, i.e., the modulus of continuity of this function satisfies the Dini condition \cite[p. 32]{AB}. A Dini-smooth arc is defined as a proper subarc of a Dini-smooth curve. Further, a Dini-smooth domain is a domain bounded by a Dini-smooth curve.

We use general discrepancy results for deterministic polynomials obtained by Andrievskii and Blatt \cite[Chapter 2]{AB} to study the expected deviation of zero counting measures for random polynomials from the limiting equilibrium measures.

\begin{theorem} \label{thm3.3}
Suppose that $E$ is a compact set with Dini-smooth arc $L\subset E$ such that the distance from $L$ to $E\setminus L$ is positive. For $P_n(z)=\sum_{k=0}^n A_k B_k(z),$ let $\{A_k\}_{k=0}^n$ be random variables satisfying $\E[|A_k|^t]<\infty,\ k=0,\ldots,n,$ for a fixed $t\in(0,1],$ and $\E[\log|A_n|] > -\infty$. Then we have for all large $n\in \N$ that
\begin{align} \label{3.3}
&\E\left[\left|(\tau_n-\mu_E)(\cA_r)\right|\right] \\ \nonumber &\le C \left[\frac{1}{n}\left(\frac{1}{t}\log \left(\sum_{k=0}^n \E[|A_k|^t]\right) + \log \frac{\max_{0 \le k \le n} \|B_k\|_{\infty}}{|b_{n,n}|(\textup{cap}(E))^n} - \E[\log |A_n|]\right)\right]^{1/2},
\end{align}
where $C>0$ depends only on $E$ and $r$. Furthermore, if $E$ is a finite union of closed intervals on the real line, then \eqref{3.3} holds true with $C=8$ and $\cA_r$ being the union of vertical strips $\{z\in\C: \Re(z)\in E\}.$
\end{theorem}

\begin{corollary} \label{cor3.4}
Let $P_n(z)=\sum_{k=0}^n A_{k,n} B_k(z),\ n\in\N,$ be a sequence of random polynomials, and let $E$ satisfy the assumptions of Theorem \ref{thm3.3}. Suppose that for $t\in(0,1]$ we have
\begin{align} \label{3.4}
\limsup_{n\to\infty} \max_{k=0,\ldots,n} \E[|A_{k,n}|^t] < \infty
\end{align}
and
\begin{align} \label{3.5}
\liminf_{n\to\infty} \E[\log|A_{n,n}|] > - \infty.
\end{align}
If the basis polynomials $B_k$ satisfy \eqref{2.3}, then 
\begin{align} \label{3.6}
\lim_{n\to\infty} \E\left[\left|(\tau_n-\mu_E)(\cA_r)\right|\right] = 0.
\end{align}
Furthermore, if \eqref{3.2} is satisfied, then
\begin{align} \label{3.7}
\E\left[\left|(\tau_n-\mu_E)(\cA_r)\right|\right] = O\left(\sqrt{\frac{\log{n}}{n}}\right) \quad \mbox{as }n\to\infty.
\end{align}
\end{corollary}

The conclusion of Corollary \ref{cor3.4} stated in \eqref{3.7} holds for the bases of orthogonal polynomials with respect to the weighted arclength measure on $E$, and of Faber polynomials when $E$ is a single arc. One only needs to verify that \eqref{3.2} is satisfied in those cases. 

\begin{corollary} \label{cor3.5}
Assume that conditions \eqref{3.4}-\eqref{3.5} hold for the coefficients.\\
\textup{(i)} Suppose that $E$ is a finite union of disjoint Dini-smooth arcs. If the basis polynomials $B_k$ are orthonormal with respect to a positive Borel measure $\mu$  such that $d\mu(s)=w(s)\,ds$, where $w(s)\ge c > 0$ for almost every point on $E$, then \eqref{3.2} is satisfied, and \eqref{3.7} holds true.\\
\textup{(ii)} Suppose that $E$ is an arbitrary Jordan arc. If the basis polynomials $B_k$ are
the Faber polynomials of $E$, then \eqref{3.2} holds true. Hence \eqref{3.7} is valid provided  $E$ is a Dini-smooth arc.
\end{corollary}

\medskip
We also give corresponding results for smooth domains (or closed curves). Suppose that $E$ is a compact set whose connected component $S$ is a closed Jordan domain such that $\textup{dist}(S,E\setminus S)>0.$ We define the ``generalized sector" $\cA_r$ by using the conformal mapping $\Phi$ from the annular region $U_S$ with inner boundary $\partial S$ to an annulus $1<|w|<R,\ R>1,$ constructed in the same way as before Theorem \ref{thm3.3}. In addition, we introduce a conformal mapping $\phi$ from the interior Jordan domain $G$ of $S$ onto the unit disk $\D$ such that $\phi(z_0)=0$ for a point $z_0\in G.$ It is known that both mappings $\Phi$ and $\phi$ extend continuously to $\partial S$, being bijections between $\partial S$ and $\T.$ For any subarc $J\subset \partial S$ and $r\in(1,r_0)$, we define 
\begin{align*}
\cA_r=\cA_r(J) &= \{z\in \overline{U}_S: 1 \le |\Phi(z)| \le r \mbox{ and }\Phi(z)/|\Phi(z)| \in \Phi(J)\} \\ &\cup \{z\in \overline{G}: 1/r \le |\phi(z)| \le 1 \mbox{ and }\phi(z)/|\phi(z)| \in \phi(J)\}. 
\end{align*}
Again, $\cA_r$ may be described as a curvilinear strip around $J$ that is bounded by the level curves $|\Phi(z)|=r$ and $|\phi(z)|=1/r,\ r>1.$

\begin{theorem} \label{thm3.6}
Suppose that $E$ is a compact set whose connected component $S$ is a closed Dini-smooth domain such that $\textup{dist}(S,E\setminus S)>0$, with an interior point $w\in S^\circ$. For $P_n(z)=\sum_{k=0}^n A_k B_k(z),$ let $\{A_k\}_{k=0}^n$ satisfy $\E[|A_k|^t]<\infty,\ k=0,\ldots,n,$ for a fixed $t\in(0,1]$. If $\E[\log|A_n P_n(w)|] > -\infty$ then we have for all large $n\in \N$ that
\begin{align} \label{3.8}
&\E\left[\left|(\tau_n-\mu_E)(\cA_r)\right|\right] \\ \nonumber &\le C \left[\frac{1}{n}\left(\frac{2}{t}\log \left(\sum_{k=0}^n \E[|A_k|^t]\right) + \log \frac{\max_{0 \le k \le n} \|B_k\|_E^2}{|b_{n,n}|(\textup{cap}(E))^n} - \E[\log|A_n P_n(w)|] \right)\right]^{1/2},
\end{align}
where $C>0$ depends only on $E$ and $r$. 

In particular, if $\E[\log|A_n|] > -\infty$, $A_0$ is independent from $A_1,\ldots,A_n,$ and $\E[\log|A_0+z|] \ge L > -\infty$ for all $z\in\C,$ then
\begin{align} \label{3.9}
\E[\log|A_n P_n(w)|] \ge \log|b_{0,0}| + \E[\log|A_n|] + L > -\infty,
\end{align}
and \eqref{3.8} holds.
\end{theorem}

If $\nu$ is the probability measure of $A_0$, then the assumption $\E[\log|A_0+z|] \ge L > -\infty$ for all $z\in\C$ may be interpreted in terms of the logarithmic potential of $\nu$ as $U^{\nu}(z) = - \int \log|t-z|\,d\nu(t) \le - L < \infty$ for all $z\in\C.$ Measures with uniformly bounded above potentials are well understood in potential theory, and they represent a wide class that do not have large local concentration of mass, e.g., they cannot have point masses. 

We next state the analog of Corollary \ref{cor3.4}.

\begin{corollary} \label{cor3.7}
Let $P_n(z)=\sum_{k=0}^n A_{k,n} B_k(z),\ n\in\N,$ be a sequence of random polynomials, and let $E$ satisfy the conditions of Theorem \ref{thm3.6}. Suppose that assumptions \eqref{3.4}, \eqref{3.5} and
\begin{align} \label{3.10}
\liminf_{n\to\infty} \inf_{z\in\C} \E[\log|A_{0,n}+z|] > - \infty
\end{align}
are satisfied for the coefficients, and that $A_{0,n}$ is independent from $\{A_{k,n}\}_{k=1}^n$ for all large $n$. If the basis polynomials $B_k$ satisfy \eqref{3.2}, then \eqref{3.7} holds true.
\end{corollary}

We give examples of typical bases satisfying \eqref{3.2} below.

\begin{corollary} \label{cor3.8}
Assume that conditions \eqref{3.4}, \eqref{3.5} and \eqref{3.10} hold for the coefficients, and that $A_{0,n}$ is independent from $\{A_{k,n}\}_{k=1}^n$ for all large $n$.\\
\textup{(i)} Suppose that $E$ is a finite union of mutually exterior closed Dini-smooth domains. If the basis polynomials $B_k$ are orthonormal with respect to a positive Borel measure $\mu$ supported on $\partial E$ such that $d\mu(s)=w(s)\,ds$, where $w(s)\ge c > 0$ for almost every point of $E$ in $ds$-sense, then \eqref{3.2} is satisfied and \eqref{3.7} holds true.\\
\textup{(ii)} Suppose that $E$ is the closure of an arbitrary Jordan domain. If the basis polynomials $B_k$ are the Faber polynomials of $E$, then \eqref{3.2} holds true. Hence \eqref{3.7} is valid provided $\partial E$ is a Dini-smooth curve.\\
\textup{(iii)} Suppose that $E$ is a finite union of mutually exterior closed Dini-smooth domains. If the basis polynomials $B_k$ are orthonormal with respect to $d\mu(z)=w(z)\,dA(z),$ where $dA$ is the area measure on $E$ and $w(z)\ge c>0$ a.e. in $dA$-sense, then \eqref{3.2} is satisfied and \eqref{3.7} holds true.\\
\end{corollary}

It is clear that if the coefficients have identical distributions, then conditions \eqref{3.4} and \eqref{3.5} reduce to those on the single coefficient $A_0.$ One can relax conditions on the orthogonality measure $\mu$ while preserving the results of Corollaries \ref{cor3.5} and \ref{cor3.8}, e.g., one can show that \eqref{3.7} also holds for polynomials orthogonal with respect to the generalized Jacobi weights of the form $w(s)=v(s)\prod_{j=1}^J |s-s_j|^{\alpha_j},$ where $v(s)\ge c > 0$ a.e., in terms of the inner product defined either by $ds$ or by $dA.$ It is also possible to significantly relax the geometric conditions on $E$, by using the discrepancy results from \cite{AB} for quasiconformal arcs and curves. Thus smoothness is not critical for the results of this section, but the square root in all discrepancy estimates should then be replaced with a different (smaller) power depending on angles at the boundary of $E$.

\section{Proofs} \label{secP}

\subsection{Proofs for Section 2}

One of the key ingredients in the study of asymptotic zero distribution of polynomials is known to be the $n$-th root limiting behavior of their coefficients, see \cite{AB} for details. We prove the following probabilistic version of such results. Let $\{X_n\}_{n=1}^{\infty}$ be a sequence of complex valued random variables, and let $F_n$ be the distribution function of $|X_n|,\ n\in\N.$ We use the assumptions on random variables $X_n$ that match those of \eqref{2.1} and \eqref{2.2} in Section 2. 

\begin{lemma}\label{lem4.1}
If there is $N\in\N$ and a decreasing function $f:[a,\infty)\to[0,1],\ a>1,$ such that
\begin{align*} 
\int_a^{\infty} \frac{f(x)}{x}\,dx < \infty \mbox{ and } 1-F_n(x)\le f(x),\ \forall\, x\in[a,\infty), 
\end{align*}
holds for all $n\ge N$, then
\begin{align} \label{4.1}
\limsup_{n\to\infty} |X_n|^{1/n} \le 1 \quad\mbox{ a.s.}
\end{align}
Further, if there is $N\in\N$ and an increasing function $g:[0,b]\to[0,1],\ 0<b<1,$ such that
\begin{align*}
\int_0^b \frac{g(x)}{x}\,dx < \infty  \mbox{ and } F_n(x)\le g(x),\ \forall\, x\in[0,b],
\end{align*}
holds for all $n\ge N$, then
\begin{align} \label{4.2}
\liminf_{n\to\infty} |X_n|^{1/n} \ge 1 \quad\mbox{ a.s.}
\end{align}
Hence if both assumptions are satisfied, then 
\begin{align} \label{4.3}
\lim_{n\to\infty} |X_n|^{1/n} = 1 \quad\mbox{ a.s.}
\end{align}
\end{lemma}

We use a standard method for finding the almost sure limits of \eqref{4.1}-\eqref{4.3} via the first Borel-Cantelli lemma stated below (see, e.g., \cite[p. 96]{Gut}).

\medskip
\noindent\textbf{Borel-Cantelli Lemma} \textit{Let $\{\cE_n\}_{n=1}^{\infty}$ be a sequence of arbitrary events. If $\sum_{n=1}^{\infty} \P(\cE_n) < \infty$ then $\P(\cE_n \mbox{ occurs infinitely often})=0.$}

\begin{proof}[Proof of Lemma \ref{lem4.1}]
We first prove \eqref{4.1}. For any fixed $\eps>0$, define events $\cE_n=\{|X_n| > e^{\eps n}\}, \ n\in\N.$ Using the first assumption and letting $m:=\max(N,\lfloor\frac{1}{\eps}\log{a}\rfloor)+2$, we obtain
\begin{align*}
\sum_{n=m}^{\infty} \P(\cE_n) &= \sum_{n=m}^{\infty} \left(1 - \P(\{|X_n|\le e^{\eps n}\}) \right) = \sum_{n=m}^{\infty} (1-F_n(e^{\eps n})) \le \sum_{n=m}^{\infty} f(e^{\eps n}) \\ &\le \int_{m-1}^{\infty} f(e^{\eps t})\,dt \le \frac{1}{\eps} \int_a^{\infty} \frac{f(x)}{x}\,dx < \infty.
\end{align*}
Hence $\P(\cE_n \mbox{ occurs infinitely often})=0$ by the first Borel-Cantelli lemma, so that the complementary event $\cE_n^c$ must happen for all large $n$ with probability 1. This means that $|X_n|^{1/n} \le e^{\eps}$ for all sufficiently large $n\in\N$ almost surely.  We obtain that
\[
\limsup_{n\to\infty} |X_n|^{1/n} \le e^{\eps} \quad\mbox{a.s.},
\]
and \eqref{4.1} follows because $\eps>0$ may be arbitrarily small.

The proof of \eqref{4.2} proceeds in a similar way. For any given $\eps>0$, we set $\cE_n=\{|X_n| \le e^{- \eps n}\}, \ n\in\N.$ Using the second assumption and letting $m:=\max(N,\lfloor -\frac{1}{\eps}\log{b}\rfloor)+2$, we have
\begin{align*}
\sum_{n=m}^{\infty} \P(\cE_n) &= \sum_{n=m}^{\infty} F_n(e^{- \eps n}) \le \sum_{n=m}^{\infty} g(e^{- \eps n}) \\ &\le \int_{m-1}^{\infty} g(e^{- \eps t})\,dt \le \frac{1}{\eps} \int_0^{b} \frac{g(x)}{x}\,dx < \infty.
\end{align*}
Hence $\P(\cE_n \mbox{ i.o.})=0$, and $|X_n|^{1/n} > e^{-\eps}$ holds for all sufficiently large $n\in\N$ almost surely.  We obtain that 
\[
\liminf_{n\to\infty} |X_n|^{1/n} \ge e^{-\eps} \quad\mbox{a.s.},
\]
and \eqref{4.2} follows by letting $\eps\to 0$.
\end{proof}

Lemma \ref{lem4.1} implies that any infinite sequence of coefficients satisfying Assumptions 1 and 2 of Section 2 must also satisfy \eqref{4.3}. We state this as follows.

\begin{lemma}\label{lem4.2}
Suppose that \eqref{2.1} and \eqref{2.2} hold for the coefficients $A_n$ of random polynomials. Then the following limits exist almost surely:
\begin{align} \label{4.4}
\lim_{n\to\infty} |A_n|^{1/n} = 1 \quad\mbox{ a.s.},
\end{align}
\begin{align} \label{4.5}
\lim_{n\to\infty} |A_k|^{1/n} = 1 \quad\mbox{ a.s.},\ k=0,1,2,\ldots,
\end{align}
and
\begin{align} \label{4.6}
\lim_{n\to\infty} \max_{0\le k\le n} |A_k|^{1/n} = 1 \quad\mbox{ a.s.}
\end{align}
\end{lemma}

\begin{proof}[Proof of Lemma \ref{lem4.2}]
Limit \eqref{4.4} follows from Lemma \ref{lem4.1} by letting $X_n=A_n,\ n\in\N$. Similarly, if we set for a fixed $k\in\N\cup\{0\}$ that $X_n=A_k,\ n\in\N,$ then \eqref{4.5} is immediate. 

We deduce \eqref{4.6} from \eqref{4.4}. Let $\omega$ be any elementary event such that 
\[
\lim_{n\to\infty} |A_n(\omega)|^{1/n} = 1,
\]
which holds with probability one. We immediately obtain that
\[
\liminf_{n\to\infty} \max_{0\le k\le n} |A_k(\omega)|^{1/n} \ge \liminf_{n\to\infty} |A_n(\omega)|^{1/n} = 1.
\]
On the other hand, elementary properties of limits imply that
\[
\limsup_{n\to\infty} \max_{0\le k\le n} |A_k(\omega)|^{1/n} \le 1.
\]
Indeed, for any $\eps>0$ there $n_{\eps}\in\N$ such that $|A_n(\omega)|^{1/n} \le 1 + \eps$ for all $n\ge n_{\eps}$ by \eqref{4.4}. Hence 
\[
\max_{0\le k\le n} |A_k(\omega)|^{1/n} \le \max\left(\max_{0\le k\le n_{\eps}} |A_k(\omega)|^{1/n}, 1+\eps \right) \to 1+\eps \quad \mbox{ as } n\rightarrow\infty,
\]
and the result follows by letting $\eps\to 0.$
\end{proof}

\medskip
We state a somewhat modified version of the result due to Blatt, Saff and Simkani \cite{BSS}, which is used to prove all equidistribution theorems of Section 2.

\noindent
{\bf Theorem BSS.} {\em Let $E\subset\C$ be a compact set, $\textup{cap}(E)>0$.  If a sequence of polynomials $P_n(z) = \sum_{k=0}^n c_{k,n} z^k$ satisfy
\begin{align} \label{4.7}
\limsup_{n\to\infty} \|P_n\|_E^{1/n} \le 1 \quad \mbox{and} \quad \lim_{n \rightarrow \infty} |c_{n,n}|^{1/n}=1/\textup{cap}(E),
\end{align}
and for any closed set $A$ in the bounded components of $\C\setminus \textup{supp}\,\mu_E$ we have
\begin{align} \label{4.8}
\lim_{n\to\infty} \tau_n(A) = 0,
\end{align}
then the zero counting measures $\tau_n$ converge weakly to $\mu_E$ as $n\to\infty.$}

It is known that \eqref{4.8} holds if every bounded component of $\C\setminus \textup{supp}\,\mu_E$ contains a compact set $K$ such that 
\begin{align} \label{4.9}
\liminf_{n\to\infty} \|P_n\|_K^{1/n} \ge 1,
\end{align}
see Bloom \cite[p. 1706]{Bl1} and \cite[p. 134]{Bl2}, and see Grothmann \cite[p. 352]{Gr} (also \cite{AB}) for the case of unbounded component of $\C\setminus \textup{supp}\,\mu_E$. In applications, this compact set $K$ is often selected as a single point.

\begin{proof}[Proof of Theorem \ref{thm2.1}]
We apply Theorem BSS with $E=\T$. Recall that $\textup{cap}(\T)=1$ and $d\mu_{\T}(e^{it})=dt/(2\pi),$ see \cite{Ra}. It is immediate that
\[
\|P_n\|_{\T} \le \sum_{k=0}^n |A_k| \le (n+1) \max_{0\le k\le n} |A_k|.
\]
Using \eqref{4.4} and \eqref{4.6} of Lemma \ref{lem4.2}, we conclude that \eqref{4.7} holds almost surely. On the other hand, \eqref{4.5} with $k=0$ also gives that
\[
\lim_{n\to\infty} |P_n(0)|^{1/n} = \lim_{n\to\infty} |A_0|^{1/n} = 1 \quad\mbox{ a.s.},
\]
meaning that \eqref{4.9} is satisfied for $K=\{0\}$ almost surely. Hence \eqref{4.8} holds a.s. for any compact subset $A$ of the unit disk, and the result follows.
\end{proof}

\begin{proof}[Proof of Theorem \ref{thm2.2}]
Since $\textup{supp}\,\mu_E\subset E$, we have that $\C\setminus \textup{supp}\,\mu_E$ has no bounded components in this case, and \eqref{4.8} of Theorem BSS holds trivially. Thus we only need to prove \eqref{4.7} for polynomials
\[
P_n(z) = \sum_{k=0}^n A_k B_k(z) = A_n b_{n,n} z^n + \ldots,\quad n\in\N.
\]
Applying \eqref{4.4} of Lemma \ref{lem4.2} and \eqref{2.3}, we obtain for their leading coefficients that
\[
\lim_{n\to\infty} |A_n b_{n,n}|^{1/n} = 1/\textup{cap}(E) \quad\mbox{ a.s.}
\]
Furthermore,
\[
\|P_n\|_E \le \sum_{k=0}^n |A_k| \|B_k\|_E \le (n+1) \max_{0\le k\le n} |A_k| \, \max_{0\le k\le n} \|B_k\|_E.
\]
Note that \eqref{2.3} implies by a simple argument (already used in the proof of Lemma \ref{lem4.2}) that
\[
\limsup_{n\to\infty} \max_{0\le k\le n} \|B_k\|_E^{1/n} \le 1.
\]
Combining this fact with \eqref{4.6} of Lemma \ref{lem4.2}, we obtain that
\[
\limsup_{n\to\infty} \|P_n\|_E^{1/n} \le 1 \quad\mbox{ a.s.}
\]
\end{proof}

\begin{proof}[Proof of Corollary \ref{cor2.3}]
Since the coefficient conditions \eqref{2.1}-\eqref{2.2} hold by our assumptions, we only need to verify that the bases satisfy \eqref{2.3} in both cases (i) and (ii). Then almost sure convergence of $\tau_n$ to $\mu_E$ will follow from Theorem \ref{thm2.2}.

(i) Our assumptions on the orthogonality measure $\mu$ and set $E$ imply that the orthogonal polynomials have regular asymptotic behavior expressed by \eqref{2.3} according to Theorem 4.1.1 and Corollary 4.1.2 of \cite[pp. 101-102]{ST}. Corollary 4.1.2 is stated for a set $E$ consisting of smooth arcs and curves, but its proof holds for arbitrary rectifiable case, because $\mu$ and $\mu_E$ are both absolutely continuous with respect to the arclength $ds$. In fact, it is known that the density of the equilibrium measure is expressed via normal derivatives of the Green function $g_E$ for the complement of $E$ from both sides of the arcs:
\[
d\mu_E=\frac{1}{2\pi}\left(\frac{\partial g_E}{\partial n_+} + \dfrac{\partial g_E}{\partial n_-}\right)\,ds,
\]
see Theorem 1.1 and Example 1.2 of \cite{Pr2}. Furthermore, $d\mu_E/ds>0$ almost everywhere in the sense of arclength on $E,$ see Garnett and Marshall \cite[Chapter II]{GM}.

(ii) Assumptions imposed on $E$ imply that $\textup{cap}(E)>0$, and that Faber polynomials are well defined. In particular, the Faber polynomials of $E$ satisfy $B_n(z)=z^n/(\textup{cap}(E))^n + \ldots,\ n=0,1,\ldots,$ by definition, see \cite[Section 2.1]{Su}. Furthermore, K\"ovari and Pommerenke \cite{KP} showed that the Faber polynomials of any compact connected set do not grow fast:
\[
\|B_n\|_E = O(n^{\alpha}) \quad\mbox{as } n\to\infty,
\]
where $\alpha<1/2.$ Hence \eqref{2.3} holds true in this case.
\end{proof}

\begin{proof}[Proof of Theorem \ref{thm2.4}]
It is known that in all three considered cases of Szeg\H{o}, Bergman and Faber bases, we have \eqref{2.3} satisfied. For the cases of Bergman and Szeg\H{o} polynomials, see pages 288-290 and pages 336-338 respectively in the book of Smirnov and Lebedev \cite{SL}. The case of Faber polynomials was considered above in the proof of Corollary \ref{cor2.3}, part (ii).  Arguing as in the proof of Theorem \ref{thm2.2}, we see that \eqref{4.7} of Theorem BSS holds true for $P_n(z) = \sum_{k=0}^n A_k B_k(z).$ Furthermore, for any compact set $K$ in the interior of $E$ denoted by $E^\circ$, we have (cf. \cite[pp. 290 and 338]{SL} and \cite[Section 2.3]{Su}) that 
\[
\limsup_{n\to\infty} \|B_n\|_K^{1/n} < 1.
\]
Since \eqref{4.4} holds with probability one, we conclude that the series $f(z)=\sum_{k=0}^{\infty} A_k B_k(z)$ converges uniformly on compact subsets of the analytic Jordan domain $E^\circ$ with probability one. Its limit is (almost surely) an analytic function $f$ that cannot vanish identically because of \eqref{4.4} and uniqueness of series expansions in Szeg\H{o}, Bergman and Faber polynomials (see \cite[pp. 293 and 340]{SL} and Section 6.3 of \cite{Su} for these facts). Hence for each limit $f$ there is a point $z_f\in E^{\circ}$ such that $f(z_f)\neq 0.$ This means $\lim_{n\to\infty} P_n(z_f) = f(z_f) \neq 0,$ so that \eqref{4.9} is satisfied with $K=\{z_f\}$. Thus \eqref{4.8} holds almost surely for any compact subset of $E^\circ$ (the only bounded component of $\overline{\C}\setminus \textup{supp}\,\mu_E = \overline{\C}\setminus \partial E$), and the result follows from Theorem BSS.
\end{proof}

\begin{proof}[Proof of Theorem \ref{thm2.5}]
We use Theorem BSS again. Condition \eqref{4.7} is verified exactly as in the proof of Theorem \ref{thm2.2}, so that we omit that argument. It remains to show that \eqref{4.8} holds almost surely as a consequence of \eqref{2.4}, which is again done via \eqref{4.9}. In particular, we prove that 
\begin{align} \label{4.10}
\liminf_{n\to\infty} |P_n(w)|^{1/n} \ge 1
\end{align}
holds almost surely for every given $w\in\C$. Define the events 
\[
\cE_n = \{|P_n(w)| \le e^{- \eps n}\} =  \left\{\frac{1}{\eps}\log^-|P_n(w)| \ge n \right\}, \quad n\in\N.
\] 
For any fixed $t>1,$ Chebyshev's inequality gives
\[
\P(\cE_n) \le \frac{1}{n^t} \E\left[ \left(\frac{1}{\eps}\log^-|P_n(w)|\right)^t \right], \quad n\in\N.
\]
Note that 
\begin{align*}
\left(\log^-|P_n(w)|\right)^t &\le \left(\log^-|b_{0,0}| + \log^-\left|A_0 + \sum_{k=1}^n \frac{A_k}{b_{0,0}} B_k(w)\right|\right)^t \\ &\le 2^t \left(\left(\log^-|b_{0,0}|\right)^t + \left(\log^-\left|A_0 + \sum_{k=1}^n \frac{A_k}{b_{0,0}} B_k(w)\right|\right)^t \right).
\end{align*}
We use independence of $A_0$ from the rest of coefficients and \eqref{2.4} to estimate
\[
\E\left[\left(\log^-\left|A_0 + \sum_{k=1}^n \frac{A_k}{b_{0,0}} B_k(w)\right|\right)^t \right] \le \sup_{z\in\C}\, \E\left[(\log^-|A_0-z|)^t\right] =: C < \infty,
\]
which gives
\[
\E\left[\left(\log^-|P_n(w)|\right)^t \right] \le 2^t \left(\left(\log^-|b_{0,0}|\right)^t + C\right).
\]
It follows that
\begin{align*}
\sum_{n=1}^{\infty} \P(\cE_n) \le \frac{2^t}{\eps^t} \left(\left(\log^-|b_{0,0}|\right)^t + C\right) \sum_{n=1}^{\infty} \frac{1}{n^t} < \infty.
\end{align*}
Hence $\P(\cE_n \mbox{ i.o.})=0$ by the first Borel-Cantelli lemma, and $|P_n(w)|^{1/n} > e^{-\eps}$ holds for all sufficiently large $n\in\N$ with probability one.  We obtain that 
\[
\liminf_{n\to\infty} |P_n(w)|^{1/n} \ge e^{-\eps} \quad\mbox{a.s.},
\]
and \eqref{4.10} follows by letting $\eps\to 0$.
\end{proof}

The following lemma serves as a substitute of Lemma \ref{lem4.2}. It is necessary for the proofs of analogs of results from Section 2 generalized under Assumptions 1* and 2*.

\begin{lemma}\label{lem4.3}
Suppose that \eqref{2.5} and \eqref{2.6} hold for the coefficients $A_{k,n}$ of random polynomials. Then the following limits exist almost surely:
\begin{align} \label{4.11}
\lim_{n\to\infty} |A_{n,n}|^{1/n} = 1 \quad\mbox{ a.s.},
\end{align}
\begin{align} \label{4.12}
\lim_{n\to\infty} |A_{k,n}|^{1/n} = 1 \quad\mbox{ a.s.},\ k\in\N\cup\{0\},
\end{align}
and
\begin{align} \label{4.13}
\lim_{n\to\infty} \max_{0\le k\le n} |A_{k,n}|^{1/n} = 1 \quad\mbox{ a.s.}
\end{align}
\end{lemma}

\begin{proof}[Proof of Lemma \ref{lem4.3}]
Limits \eqref{4.11} and \eqref{4.12} follow from Lemma \ref{lem4.1} by correspondingly letting $X_n=A_{n,n},\ n\in\N$, and $X_n=A_{k,n},\ n\in\N,$ for a fixed $k\in\N\cup\{0\}$. In fact, this argument holds under weaker assumptions such as \eqref{2.1} and \eqref{2.2}, and does not require independence of coefficients.

In order to prove \eqref{4.13}, we introduce the random variable $Y_n=\max_{0\le k\le n} |A_{k,n}|$, and denote its distribution function by $F_n(x),\ n\in\N.$ Note that 
\[
\liminf_{n\to\infty} |Y_n|^{1/n} \ge \liminf_{n\to\infty} |A_{n,n}|^{1/n} = 1 \quad\mbox{a.s.}
\]
Using independence of $|A_{k,n}|,\ k=0,\ldots,n,$ for each $n\ge N$, and applying \eqref{2.5}, we estimate
\begin{align*}
F_n(x) = \prod_{k=0}^n F_{k,n}(x) \ge (1-f(x))^{n+1} \ge 1-(n+1)f(x), \quad x\ge a.
\end{align*}
For any fixed $\eps>0$, define events $\cE_n=\{|Y_n| > e^{\eps n}\}, \ n\in\N.$ Letting $m:=\max(N,\lfloor\frac{1}{\eps}\log{a}\rfloor)+2$, we obtain from the above estimate and \eqref{2.5} that
\begin{align*}
\sum_{n=m}^{\infty} \P(\cE_n) &= \sum_{n=m}^{\infty} \left(1 - \P(\{|Y_n|\le e^{\eps n}\}) \right) = \sum_{n=m}^{\infty} (1-F_n(e^{\eps n})) \le \sum_{n=m}^{\infty} (n+1)f(e^{\eps n}) \\ &\le 2 \int_{m-1}^{\infty} t\,f(e^{\eps t})\,dt \le \frac{2}{\eps} \int_a^{\infty} \frac{f(x)\log{x}}{x}\,dx < \infty.
\end{align*}
Hence $\P(\cE_n \mbox{ i.o.})=0$ by the first Borel-Cantelli lemma, and $|Y_n|^{1/n} \le e^{\eps}$ for all sufficiently large $n\in\N$ almost surely.  We obtain that
\[
\limsup_{n\to\infty} |Y_n|^{1/n} \le e^{\eps} \quad\mbox{a.s.},
\]
and \eqref{4.13} follows after letting $\eps\to 0.$
\end{proof}

\subsection{Proofs for Section 3}

The following lemma is used several times below.

\begin{lemma} \label{lem4.4}
If $A_k,\ k=0,\ldots,n,$ are complex random variables satisfying $\E[|A_k|^t]<\infty,\ k=0,\ldots,n,$ for a fixed $t\in(0,1],$ then
\begin{align} \label{4.14}
\E\left[\log\sum_{k=0}^n |A_k|\right] \le \frac{1}{t}\log \left(\sum_{k=0}^n \E[|A_k|^t]\right).
\end{align}
\end{lemma}

\begin{proof}
We first state an elementary inequality. If $x_i \ge 0,\ i=0,\ldots,n,$ and $\sum_{i=0}^n x_i = 1,$ then
\[ 
\sum_{i=0}^n (x_i)^t \ge \sum_{i=0}^n x_i = 1
\]
for $t\in(0,1]$. Applying this inequality with $x_i = |A_i|/\sum_{k=0}^n |A_k|,$ we obtain that
\begin{align*}
\left(\sum_{k=0}^n |A_k|\right)^t \le \sum_{k=0}^n |A_k|^t
\end{align*}
and
\begin{align*}
\E\left[\log\sum_{k=0}^n |A_k|\right] \le \frac{1}{t}\E\left[\log\left(\sum_{k=0}^n |A_k|^t\right)\right].
\end{align*}
Jensen's inequality and linearity of expectation now give that
\begin{align*}
\E\left[\log\sum_{k=0}^n |A_k|\right] &\le \frac{1}{t}\log \E\left[\sum_{k=0}^n |A_k|^t\right] = \frac{1}{t}\log \left(\sum_{k=0}^n \E[|A_k|^t]\right).
\end{align*}
\end{proof}

\begin{proof}[Proof of Theorem \ref{thm3.1}]
We use the following version of the discrepancy theorem due to Erd\H{o}s and Tur\'an stated in Proposition 2.1 of \cite{PS} (see also \cite{ET}, \cite{Ga} and \cite{AB}):
\begin{align*}
\left| \tau_n\left(\cA_r(\alpha,\beta)\right) - \frac{\beta-\alpha}{2\pi}\right| \leq
\sqrt{\frac{2\pi}{{\bf k}}} \sqrt{\frac{1}{n}\,\log\frac{\|P_n\|_{\T}}{\sqrt{|A_0A_n|}}} + \frac{2}{n(1-r)} \, \log\frac{\|P_n\|_{\T}}{\sqrt{|A_0A_n|}}.
\end{align*}
Applying Jensen's inequality, we obtain that
\begin{align*}
\E\left[\left| \tau_n\left(\cA_r(\alpha,\beta)\right) - \frac{\beta-\alpha}{2\pi}\right|\right] &\leq
\sqrt{\frac{2\pi}{{\bf k}}} \sqrt{\frac{1}{n}\, \E\left[ \log\frac{\|P_n\|_{\T}}{\sqrt{|A_0 A_n|}} \right]} + \frac{2}{n(1-r)} \, \E\left[\log\frac{\|P_n\|_{\T}}{\sqrt{|A_0 A_n|}}\right] \\ &\leq
C_r \sqrt{\frac{1}{n}\, \E\left[ \log\frac{\|P_n\|_{\T}}{\sqrt{|A_0 A_n|}} \right]},
\end{align*}
where the last inequality holds for all sufficiently large $n\in\N.$ Since $\|P_n\|_{\infty} \le \sum_{k=0}^n |A_k|,$ we use the linearity of expectation and \eqref{4.14} to estimate
\begin{align*}
\E\left[ \log\frac{\|P_n\|_{\T}}{\sqrt{|A_0 A_n|}} \right] &\le \E\left[\log \sum_{k=0}^n |A_k|\right] - \frac{1}{2}\E[\log|A_0A_n|] \\ &\le \frac{1}{t}\log \left(\sum_{k=0}^n \E[|A_k|^t]\right) - \frac{1}{2}\E[\log|A_0A_n|].
\end{align*}
The latter bound is finite by our assumptions.
\end{proof}

\begin{proof}[Proof of Corollary \ref{cor3.2}]
The result follows immediately upon using the uniform bounds $M$ and $L$ in estimate \eqref{3.1}.
\end{proof}

\begin{proof}[Proof of Theorem \ref{thm3.3}]
Note that the leading coefficient of $P_n$ is $A_n b_{n,n}$. Theorem 4.2 in Chapter 2 of \cite[p. 80]{AB} gives a discrepancy estimate of the form
\begin{align} \label{4.15}
\left|(\tau_n-\mu_E)(\cA_r)\right| \le C \sqrt{\frac{1}{n}\, \log\frac{\|P_n\|_E}{|A_n b_{n,n}| (\textup{cap}(E))^n}},
\end{align}
where constant $C$ depends only on $E$ and $r.$ Using this estimate and Jensen's inequality, we obtain that 
\begin{align*}
\E\left[\left|(\tau_n-\mu_E)(\cA_r)\right|\right] &\leq C \sqrt{\frac{1}{n}\, \E\left[ \log\frac{\|P_n\|_E}{|A_n b_{n,n}| (\textup{cap}(E))^n} \right]} \\ 
&\leq C \sqrt{\frac{1}{n}\,\left( \E[\log\|P_n\|_E] - \log(|b_{n,n}|(\textup{cap}(E))^n) - \E[\log|A_n|] \right)}.
\end{align*}
It is clear that
\[
\|P_n\|_E \le \sum_{k=0}^n |A_k| \|B_k\|_E \le \max_{0 \le k \le n} \|B_k\|_E \sum_{k=0}^n |A_k|.
\]
Hence \eqref{4.14} yields
\begin{align*}
\E\left[\log \|P_n\|_E\right] &\le \E\left[\log \sum_{k=0}^n |A_k|\right] + \log \max_{0 \le k \le n} \|B_k\|_E \\ &\le \frac{1}{t}\log \left(\sum_{k=0}^n \E[|A_k|^t]\right) + \log \max_{0 \le k \le n} \|B_k\|_E.
\end{align*}
Thus \eqref{3.3} follows by combining the above estimates.

When $E$ is a finite union of closed non-intersecting intervals, one needs to use the discrepancy estimate of Theorem 5.1 in Chapter 2 of \cite[p. 86]{AB}, which has the same form as \eqref{4.15} but with $C=8$ and $\cA_r$ being the union of vertical strips $\{z\in\C: \Re(z)\in E\}.$ The rest of the proof remains identical.
\end{proof}

\begin{proof}[Proof of Corollary \ref{cor3.4}]
We estimate the right hand side of \eqref{3.3}. For this purpose, we make two immediate observation that \eqref{3.4} implies
\[
\frac{1}{tn}\log \left(\sum_{k=0}^n \E[|A_{k,n}|^t]\right) \le O\left(\frac{\log{n}}{n}\right)\quad \mbox{as }n\to\infty,
\]
while \eqref{3.5} implies
\[
-\frac{1}{n}\,\E[\log|A_{n,n}|] \le O\left(\frac{1}{n}\right)\quad \mbox{as }n\to\infty.
\]
If \eqref{2.3} is satisfied, then 
\[
\limsup_{n\to\infty} \left(\max_{0 \le k \le n} \|B_k\|_E\right)^{1/n} = \limsup_{n\to\infty} \left(\|B_k\|_E\right)^{1/n} \le 1,
\]
and therefore
\[
\limsup_{n\to\infty} \frac{1}{n}\log \frac{\max_{0 \le k \le n} \|B_k\|_{\infty}}{|b_{n,n}|(\textup{cap}(E))^n} \le 0.
\]
Hence \eqref{3.6} follows from \eqref{2.3}, \eqref{3.3} and the above inequalities. On the other hand, if \eqref{3.2} is satisfied, then
\[
\frac{1}{n}\log \frac{\max_{0 \le k \le n} \|B_k\|_{\infty}}{|b_{n,n}|(\textup{cap}(E))^n} \le O\left(\frac{\log{n}}{n}\right)\quad \mbox{as }n\to\infty,
\]
and \eqref{3.7} follows in the same manner.
\end{proof}

\begin{proof}[Proof of Corollary \ref{cor3.5}]
In both cases, we need to verify that \eqref{3.2} is satisfied, and then apply Corollary \ref{cor3.4} to conclude that \eqref{3.7} holds. 

(i) The leading coefficient $b_{n,n}$ of the orthonormal polynomial $B_n$ provides the solution of the following extremal problem \cite[p. 14]{ST}:
\[ 
|b_{n,n}|^{-2} = \inf\left\{ \int |Q_n|^2\,d\mu \ : \ Q_n \mbox{ is a monic polynomial of degree }n \right\}.
\]
We use a monic polynomial $Q_n(z)$ that satisfies $\|Q_n\|_E \le C_1 (\textup{cap}(E))^n,$ where $C_1>0$ depends only on $E$. Existence of such polynomial for a set $E$ composed of finitely many smooth arcs and curves was first proved by Widom \cite{Wi} (see also Totik \cite{To}). Andrievskii \cite{An} recently obtained much more general results for unions of arcs and curves that are not necessarily smooth. We estimate that
\begin{align*} 
|b_{n,n}| \ge \left(\int |Q_n|^2\,d\mu \right)^{-1/2} \ge \left(\mu(E)\right)^{-1/2} \|Q_n\|_E^{-1} \ge C_1^{-1} \left(\mu(E)\right)^{-1/2} (\textup{cap}(E))^{-n}.
\end{align*}
Thus the second part of \eqref{3.2} is proved. For the proof of the first part, we apply the Nikolskii type inequality (see Theorem 1.1 of \cite{Pr1} and comments on page 689):
\[
\|B_n\|_E \le C_2 n \left(\int_E |B_n|^2\,ds\right)^{1/2} \le \frac{C_2}{\sqrt{c}}\, n \left(\int_E |B_n|^2\,w(s)ds\right)^{1/2} = \frac{C_2}{\sqrt{c}}\, n.
\]
We also used that $B_n$ is orthonormal with respect to $d\mu(s)=w(s)ds$ on the last step.

(ii) In fact, \eqref{3.2} was already verified for the Faber polynomials of any compact connected set $E$ in the proof of Corollary \ref{cor2.3}. Recall that the Faber polynomials of $E$ have the form $F_n(z)=z^n/(\textup{cap}(E))^n + \ldots,\ n=0,1,\ldots,$ by definition, see \cite{Su}. Furthermore, $\|F_n\|_E = O(n^{\alpha})\ \mbox{as } n\to\infty,$ where $\alpha<1/2,$ by \cite{KP}.
\end{proof}

\begin{proof}[Proof of Theorem \ref{thm3.6}]
This proof is similar to that of Theorem \ref{thm3.3}. Observe that the leading coefficient of $P_n$ is $A_n b_{n,n}$. Let $\cA_r$ be a ``generalized curvilinear sector" (neighborhood)  associated with a subarc $J$ of $\partial S.$ We use Theorem 4.5 from Chapter 2 of \cite[p. 85]{AB} for the needed discrepancy estimate:
\begin{align} \label{4.16}
\left|(\tau_n-\mu_E)(\cA_r)\right| \le C \sqrt{\frac{1}{n}\, \log\frac{\|P_n\|_E}{|A_n b_{n,n}| (\textup{cap}(E))^n} + \frac{1}{n}\, \log\frac{\|P_n\|_E}{|P_n(w)|}},
\end{align}
where constant $C$ depends only on $E$ and $r.$ We again apply Jensen's inequality to obtain that 
\begin{align*}
\E\left[\left|(\tau_n-\mu_E)(\cA_r)\right|\right] &\leq C \sqrt{\frac{1}{n}\, \E\left[ \log\frac{\|P_n\|_E}{|A_n b_{n,n}| (\textup{cap}(E))^n} \right] + \frac{1}{n}\,\E\left[ \log\frac{\|P_n\|_E}{|P_n(w)|}\right]}.
\end{align*}
It follows exactly as in the proof of Theorem \ref{thm3.3} that
\begin{align*}
\E\left[\log \|P_n\|_E\right] \le \frac{1}{t}\log \left(\sum_{k=0}^n \E[|A_k|^t]\right) + \log \max_{0 \le k \le n} \|B_k\|_E.
\end{align*}
and
\begin{align*}
\E\left[ \log\frac{\|P_n\|_E}{|A_n b_{n,n}| (\textup{cap}(E))^n} \right] \le \frac{1}{t}\log \left(\sum_{k=0}^n \E[|A_k|^t]\right) + \log \frac{\max_{0 \le k \le n} \|B_k\|_E}{|b_{n,n}|(\textup{cap}(E))^n} - \E[\log|A_n|].
\end{align*}
Hence \eqref{3.8} follows as combination of the above estimates.

We now proceed to the lower bound for the expectation of $\log|A_n P_n(w)|$ in \eqref{3.9} by estimating that
\begin{align*}
\E[\log|A_n P_n(w)|] &= \E\left[\log \left|A_n \sum_{k=0}^n A_k B_k(w)\right|\right]\\
&= \E[\log|A_n|] + \E\left[\log \left|\sum_{k=0}^n A_k B_k(w)\right|\right]\\
&= \E[\log|A_n|] + \log|b_{0,0}| + \E\left[\log \left| A_0 + \sum_{k=1}^n A_k \frac{B_k(w)}{b_{0,0}}\right|\right]\\
&\ge \log|b_{0,0}| + \E[\log|A_n|] + L > -\infty,
\end{align*}
where we used that $b_{0,0} \neq 0$ and that 
\[
\E\left[\log \left| A_0 + \sum_{k=1}^n A_k \frac{B_k(w)}{b_{0,0}}\right|\right] \ge \inf_{z\in\C} \E[\log|A_0+z|] \ge L
\]
by independence of $A_0$ from $\{A_k\}_{k=1}^n$.
\end{proof}

\begin{proof}[Proof of Corollary \ref{cor3.7}]
We use \eqref{3.8} and proceed in the same way as in the proof of Corollary \ref{cor3.4}. Thus \eqref{3.4} implies that
\[
\frac{2}{tn}\log \left(\sum_{k=0}^n \E[|A_{k,n}|^t]\right) \le O\left(\frac{\log{n}}{n}\right)\quad \mbox{as }n\to\infty,
\]
and \eqref{3.5} implies that
\[
-\frac{1}{n}\,\E[\log|A_{n,n}|] \le O\left(\frac{1}{n}\right)\quad \mbox{as }n\to\infty.
\]
Moreover, our assumption \eqref{3.2} about the basis again gives
\[
\frac{1}{n}\log \frac{\max_{0 \le k \le n} \|B_k\|_E^2}{|b_{n,n}|(\textup{cap}(E))^n} \le O\left(\frac{\log{n}}{n}\right)\quad \mbox{as }n\to\infty.
\]
The new component in this proof is added by \eqref{3.10}:
\begin{align*}
-\frac{1}{n}\,\E[\log|P_n(w)|] &= -\frac{1}{n}\, \E\left[\log \left|\sum_{k=0}^n A_{k,n} B_k(w)\right|\right]\\
&= -\frac{1}{n}\left(\log|b_{0,0}| + \E\left[\log \left| A_{0,n} + \sum_{k=1}^n A_{k,n} \frac{B_k(w)}{b_{0,0}}\right|\right]\right)\\ &\le O\left(\frac{1}{n}\right)\quad \mbox{as }n\to\infty.
\end{align*}
Hence \eqref{3.7} holds in the settings of Corollary \ref{cor3.7}.
\end{proof}

\begin{proof}[Proof of Corollary \ref{cor3.8}]
All parts of Corollary \ref{cor3.8} follow from Corollary \ref{cor3.7} provided we show that the corresponding bases satisfy \eqref{3.2}. But for parts (i) and (ii) this is done by the arguments essentially identical to those of proofs for parts (i) and (ii) of Corollary \ref{cor3.5}. Hence we do not repeat them.

(iii) The proof of this part is also similar to that of part (i) of Corollary \ref{cor3.5}. The leading coefficient $b_{n,n}$ of the orthonormal polynomial $B_n$ satisfies \cite[p. 14]{ST}:
\[ 
|b_{n,n}|^{-2} = \inf\left\{ \int |Q_n|^2\,d\mu \ : \ Q_n \mbox{ is a monic polynomial of degree }n \right\}.
\]
To prove the second part of \eqref{3.2}, we again use a monic polynomial $Q_n(z)$ that satisfies $\|Q_n\|_E \le C_1 (\textup{cap}(E))^n,$ see \cite{Wi}, \cite{To} and \cite{An}. It follows that
\begin{align*} 
|b_{n,n}| \ge \left(\int |Q_n|^2\,d\mu \right)^{-1/2} \ge \left(\mu(E)\right)^{-1/2} \|Q_n\|_E^{-1} \ge C_1^{-1} \left(\mu(E)\right)^{-1/2} (\textup{cap}(E))^{-n}.
\end{align*}
The first part of \eqref{3.2} follows from the area Nikolskii type inequality (see Theorem 1.3 of \cite{Pr1} and remark (i) on page 689):
\[
\|B_n\|_E \le C_2 n \left(\int_E |B_n|^2\,dA\right)^{1/2} \le \frac{C_2}{\sqrt{c}}\, n \left(\int_E |B_n|^2\,w\,dA\right)^{1/2} = \frac{C_2}{\sqrt{c}}\, n,
\]
where we used that the weighted area $L_2$ norm of $B_n$ is equal to 1 by definition.
\end{proof}

\textbf{Acknowledgement.} This research was partially supported by the National Security Agency (grant H98230-12-1-0227) and by the AT\&T Foundation.

\noindent Igor E.~Pritsker \\
Department of Mathematics\\ 
Oklahoma State University \\
Stilwater, OK 74078, USA \\
igor@math.okstate.edu

\end{document}